\documentclass[11pt,reqno,oneside]{amsart}
 \usepackage{latexsym}
 \usepackage{amssymb}
 \usepackage{amsfonts}
\usepackage{graphics}
\usepackage{mathcomp}
\author{Pierre Mathonet}
\address[Pierre Mathonet]{Mathematics Research Unit, FSTC, University of Luxembourg \\
6, rue Coudenhove-Kalergi, L-1359 Luxembourg, Luxembourg}
\email{pierre.mathonet[at]uni.lu,P.Mathonet[at]ulg.ac.be }
\author{Fabian Radoux} 
\address[Fabian Radoux]{University of Li\`ege, Institute of mathematics,\\
 Grande Traverse, 12 - B37, B-4000 Li\`ege, Belgium}
\email{Fabian.Radoux[at]ulg.ac.be}
\date{\today} 
\title[Projectively equivariant quantizations]{Projectively equivariant
quantizations over the superspace $\R^{p|q}$}

\theoremstyle{plain}
\newtheorem{theorem}{Theorem}[subsection]
\newtheorem{lemma}[theorem]{Lemma}
\newtheorem{proposition}[theorem]{Proposition}

\theoremstyle{definition}
\newtheorem{definition}[theorem]{Definition}

\theoremstyle{remark}

\newtheorem{remark}[theorem]{Remark}

\newcommand{\sla}{\mathfrak{sl}}
\newcommand{\gl}{\mathfrak{gl}}
\newcommand{\str}{\mathrm{str}}
\newcommand{\Id}{\mathrm{Id}}

\newcommand{\R}{\mathbb{R}}
\newcommand{\Z}{\mathbb{Z}}

\newcommand{\N}{\mathbb{N}}

\renewcommand{\L}{\mathcal{L}}

\renewcommand{\S}{\mathcal{S}}
\newcommand{\F}{\mathcal{F}}
\newcommand{\g}{\mathfrak{g}}

\newcommand{\h}{\mathfrak{h}}

\newcommand{\dive}{\mathrm{div}}
\newcommand{\euler}{\mathcal{E}}

\newcommand{\Vect}{\mathrm{Vect}}

\newcommand{\cc}{{\mathcal{C}}}
\begin{document}
\begin{abstract}
We investigate the concept of projectively equivariant quantization in the
framework of super projective geometry. When the projective superalgebra $\mathfrak{pgl}(p+1|q)$ is simple, our result is similar to the classical one in the purely even case: we prove the existence and uniqueness of the quantization except in some critical situations. When the projective superalgebra is not simple (i.e. in the case of $\mathfrak{pgl}(n|n)\not\cong \mathfrak{sl}(n|n)$), we show the existence of a one-parameter family of equivariant quantizations. We also provide explicit formulas in terms of a generalized divergence operator acting on supersymmetric tensor fields. 
\end{abstract}
\maketitle
MSC(2010): 17B66, 58A50
 
{\bf Keywords:} Projective superspace, differential operators, quantization.
\section{Introduction}

Denote by $\S_\delta$ the space of contravariant symmetric tensor fields with
coefficients in $\delta$-densities over $\R^n$ and by ${\mathcal
D}_{\lambda,\lambda+\delta}$ the space of differential operators acting
between spaces of densities of weight $\lambda$ and $\lambda+\delta$.
A \emph{projectively equivariant quantization} procedure over $\R^n$, as introduced
by P. Lecomte and V. Ovsienko in \cite{LO}, is a linear bijection from
$\S_\delta$ to ${\mathcal D}_{\lambda,\lambda+\delta}$ that intertwines
the actions of the Lie algebra $\mathfrak{pgl}(n+1)$ of infinitesimal projective
transformations on both spaces. This bijection is required to satisfy a
natural normalization condition (see Formula (\ref{norma})).

In their seminal work \cite{LO}, P. Lecomte and V. Ovsienko showed the existence and
uniqueness of the projectively equivariant quantization in the case $\delta=0$.
This result was generalized in \cite{DO} for arbitrary $\delta\in \R\setminus C$,
where $C$ is a set of \emph{critical values}.

Various generalizations of these results were considered so far. 
Let us mention \cite{DLO,Lecras,BHMP,BM} for the analysis of equivariant quantizations over Euclidean spaces, \cite{Bou1,Bou2,DO1,Loubon} for first steps towards the definition of such quantizations over arbitrary manifolds, \cite{Leconj} for the formal definition of \emph{natural and projectively invariant} quantization procedures over arbitrary manifolds and finally \cite{Bor,Hansoul,MR,Fox,sarah,MR1,MR2,MR3,Radoux06,Radoux08,Radoux09,CapSil} for existence theorems (for non-critical situations) and several independent constructions for such quantization procedures in more and more general contexts.

Recently, several papers dealt with the problem of equivariant quantizations in the
context of supergeometry: first, in \cite{GarMelOvs07,Mel09} the problem of equivariant quantizations over the supercircles
$S^{1|1}$ and $S^{1|2}$ was considered (with respect to orthosymplectic superalgebras). Second, the thesis \cite{Mic09}
dealt with \emph{conformally equivariant quantizations} over supercotangent bundles. 
Finally, a result extending to supergeometry the theory of projectively invariant
quantization appeared in \cite{Geo09} where it was shown that, given
a supermanifold endowed with a projective class of superconnections, it is possible
to associate a differential operator of order two to any supersymmetric tensor field
in a projectively invariant way. 

The similarity between the formula appearing in
that paper (see \cite[Theorem 3.2]{Geo09}) and the explicit formulas in the context of
$\mathfrak{pgl}(n+1)$-equivariant quantization over $\R^n$ led us to consider here
the problem of equivariant quantizations with respect to the Lie superalgebra
$\mathfrak{pgl}(p+1|q)$ over the superspace $\R^{p|q}$.

Over fields of characteristic 0, the projective Lie algebra $\mathfrak{pgl}(n)$ is isomorphic to the special linear algebra $\sla(n)$. This is not always the case if the characteristic of the ground field is positive or for Lie superalgebras. It turns out that our results vary in accordance with the isomorphism of these Lie superalgebras: 
\begin{itemize}
 \item When the projective superalgebra $\mathfrak{pgl}(p+1|q)$ is isomorphic to the special linear Lie superalgebra, our result is similar to the classical one \cite{LO,DO}: we show the existence and uniqueness of the quantization, except for a countable set of critical values of the parameter $\delta$ (see Proposition \ref{gencrit} and Theorems \ref{flatex} and \ref{Expl});
 \item When the projective superalgebra is $\mathfrak{pgl}(n|n)$ for some $n$, it is not isomorphic to $\sla(n|n)$ and our results do not depend on the parameters $\lambda$ and $\delta$. We prove the existence of a one-parameter family of $\mathfrak{psl}(n|n)$-equivariant quantizations, that turn out to be $\mathfrak{pgl}(n|n)$-equivariant (see Theorems \ref{expl2} and \ref{expl3}).
\end{itemize}
In both situations we obtain explicit formulas for the quantization in terms of a generalized divergence
operator acting on supersymmetric tensor fields. These formulas coincide with the one given by J. George \cite{Geo09} for tensor fields of degree two over
$\R^{p|q}$ endowed with the flat superconnection.

\section{Notation and problem setting}\label{tens}
In this section, we will recall the definitions of the spaces of
differential operators acting on densities and of their corresponding
spaces of symbols over the superspace $\R^{p|q}$. Several objects that we use appear
here and there in the literature, sometimes with different sign conventions. We
present them in an explicit way to fix the notation and for the convenience of the
reader. Then we set the problem of existence of projectively equivariant
quantizations and symbol maps. Throughout the paper, we only consider Lie superalgebras over $\R$. We denote by $\tilde{a}$ the
parity of an homogeneous object $a$. For indices $i\in\{1,\ldots,p+q\}$, we set
$\tilde{i}=0$ if $i\leqslant p$ and 1 otherwise. We denote by $x^1,\ldots,x^p$ the set of
even indeterminates, and by $\theta^1,\ldots,\theta^q$ the set of odd (anticommuting)
indeterminates. We also use the unified notation $y^i$, where $i\leqslant p+q$, for the set of even and odd
indeterminates, $y^1,\ldots,y^p$ being the even $x^1,\ldots,x^p$.  
\subsection{Densities and weighted symmetric tensors}
In the purely even situation, $\lambda$-densities ($\lambda\in\R$) over a manifold $M$ are defined as smooth sections of the rank 1 bundle $|\mathrm{Vol}(M)|^\lambda\to M$. This endows the space of $\lambda$-densities $\F_\lambda(M)$ with a natural action of diffeomorphisms and of vector fields. The $\mathrm{Vect}(M)$-modules $\F_\lambda(M)$ can also be seen as deformations of the space of smooth functions. 
We adopt this second point of view for the definition of densities over $\R^{p|q}$ and use the classical divergence over $\R^{p|q}$ in order to define such a deformation of the $\Vect(\R^{p|q})$-module of smooth functions $\F=C^\infty(\R^{p|q})$. 
\begin{definition}\label{Defdiv}
 The divergence of a vector field $X\in \Vect(\R^{p|q})$ is given by
\[\dive(X)=\sum_{i=1}^{p+q}(-1)^{\tilde{y_i}\widetilde{X^i}}\partial_{y^i}X^i,\]
whenever $X=\sum_{i=1}^{p+q}X^i\partial_{y^i}$.
\end{definition}

\begin{definition}
The $\Vect(\R^{p|q})$-module $\mathcal{F}_{\lambda}$ of densities of degree $\lambda$ is the space of smooth functions
$\F$ endowed with the Lie derivative given by
\[L_X^\lambda f=X(f)+\lambda\,\dive(X)f\]
for all $X\in \Vect(\R^{p|q})$ and $f\in \F$. 
\end{definition}
It turns out that the Lie derivative of densities is a particular case of Lie
derivative on generalized tensors considered for instance in \cite{BerLei81Ser,GrozmanLeites}. If $(V,\rho)$ is a representation of the algebra $\gl(p|q)$, the space of tensor fields of type $(V,\rho)$ is defined as
$T(V)=\F\otimes V$ and the Lie derivative of a tensor field $f\otimes v$
along any vector field $X$ is defined by the explicit formula
\begin{equation}\label{eqRadouxLeites}L_{X}(f\otimes v)=X(f)\otimes
v+(-1)^{\widetilde{X}\tilde{f}}\sum_{ij}f
J_{i}^{j}\otimes\rho(e_{j}^{i})v,\end{equation}
where $J_{i}^{j}=(-1)^{\tilde{y^i}\widetilde{X}+1}(\partial_{y^{i}}X^{j})$ and
$e_{j}^{i}$ is the operator defined in the standard basis $e_1,\ldots, e_{p+q}$ of $\R^{p|q}$ (see Definition \ref{canbases}) by $e_{j}^{i}(e_k)=\delta^i_ke_j$. This formula defines a representation of $\Vect(\R^{p|q})$ on $T(V)$. For example, identifying the vector $e_k\in \R^{p|q}$ with the constant vector field $\partial_{y^k}$ and taking for $\rho$ the identity representation, we get the Lie derivative (bracket) of vector fields.
\begin{remark} Formula (\ref{eqRadouxLeites}) is not word for word the one given in \cite{BerLei81Ser}. Let us explain the correspondence. In that paper, the formula is given by
\begin{equation}\label{eqLeites}L_{X}(f v)=X(f)
v+(-1)^{\widetilde{X}\tilde{f}}\sum_{ij}f
D^{ij}\otimes\rho(E_{ij})v,\end{equation}
where \[D^{ij} = (-1)^{\tilde{y^i}(\widetilde{X^j}+1)}\partial_{y^i}X^j=-(-1)^{\tilde{y^i}(\tilde{y^j}+1)}J^j_i,\] and $E_{ij}$ is the matrix whose canonical action on $e_{k}$ corresponds to the adjoint action of the vector field $y^{i}\partial_{y^j}$ on $\partial_{y^k}$ (identifying $e_k$ and $\partial_{y^k}$ as above). We then have ${E_{ij}=-(-1)^{\tilde{y^i}(\tilde{y^j}+1)} e^i_j}$, so that  Formulas (\ref{eqRadouxLeites}) and (\ref{eqLeites}) define the same Lie derivative.
\end{remark}
The particular case of densities corresponds to a
vector space $B^\lambda$ of dimension $1|0$ spanned by one element $u$ and the
representation of $\gl(p|q)$ defined by $\rho(A)u=-\lambda\,\str(A)u.$ Note that
in the formulas below, we will not explicitly write down the generator $u$ of
$B^\lambda$ unless this leads to confusion.

As we continue, we will be interested in the space of symmetric tensor fields over
$\R^{p|q}$. 
Recall that the symmetric algebra over $\R^{p|q}$ is defined as the quotient of the
tensor algebra of $\R^{p|q}$ by the graded ideal generated by elements of the form
$v\otimes w-(-1)^{\tilde{v}\tilde{w}}w\otimes v$, where $v,w\in\R^{p|q}$. There is a bijection from the
symmetric algebra to the subspace of the tensor algebra made of symmetric tensors.
It is given, for homogeneous elements $v_1,\ldots,v_k\in \R^{p|q}$, by
\[
[v_1\otimes\cdots\otimes v_k]\mapsto v_1\vee\cdots\vee v_k=\sum_{\sigma\in
S_k}\mathrm{sgn}(\sigma,v_1,\ldots,v_k) v_{\sigma^{-1}(1)}\otimes\cdots\otimes v_{\sigma^{-1}(k)},
\]
where $\mathrm{sgn}(\sigma,v_1,\ldots,v_k)$ is the signature of the permutation $\sigma'$ induced by $\sigma$ on the ordered subset of all odd elements among $v_{1},\ldots,v_{k}$. We denote by $S^k$ the subspace of elements of degree $k$ in the symmetric algebra of the superspace $\R^{p|q}$ or the corresponding space of supersymmetric tensors.

There is also a natural representation $\rho$ of $\gl(p|q)$ on $S^k$
defined by 
\begin{equation}\label{rhotens}
\rho(A)(v_{1}\vee\cdots\vee
v_{l})=\sum_{j=1}^{l}(-1)^{\tilde{A}(\sum_{k=1}^{j-1}\tilde{v}_{k})}v_{1}\vee\cdots\vee
A v_{j}\vee\cdots\vee v_{l}.\end{equation}
The space of \emph{weighted symmetric tensors} $S^k_\delta$ over $\R^{p|q}$ is then
the tensor product $B^\delta\otimes S^k$, where the weight $\delta$ is a real number. It is endowed with the tensor
product representation of $\gl(p|q)$. 
\begin{definition}\label{defsymb}
 The space of weighted symmetric tensor fields $\S^k_{\delta}$ is the space of
tensor fields $T(S^k_\delta)=\F\otimes
S^k_\delta$ endowed with the action $L$ of $\Vect(\R^{p|q})$
defined by Formula (\ref{eqRadouxLeites}).
\end{definition}

\subsection{Differential operators and symbols}
We denote by $\mathcal{D}_{\lambda,\mu}$ the space of 
linear differential operators from $\mathcal{F}_\lambda$ to $\mathcal{F}_\mu$. The
space $\mathcal{D}_{\lambda,\mu}$ is filtered by the order of
differential operators. We denote by $\mathcal{D}^k_{\lambda,\mu}$ the
space of differential operators of order at most $k$. 
It is easy to see that every $D\in \mathcal{D}^k_{\lambda,\mu}$ has a unique
expression of the form
\begin{equation}\label{eqDiffop}
 D(f)=\sum_{|\alpha|\leqslant k}f_\alpha(x,\theta)\genfrac{(}{)}{}{}{\partial}{\partial
x^1}^{\alpha_1}\cdots\genfrac{(}{)}{}{}{\partial}{\partial
x^p}^{\alpha_p}\genfrac{(}{)}{}{}{\partial}{\partial
\theta^1}^{\alpha_{p+1}}\cdots\genfrac{(}{)}{}{}{\partial}{\partial
\theta^p}^{\alpha_{p+q}}f,
\end{equation}
where $\alpha$ is a multiindex, $|\alpha|=\sum_{i=1}^{p+q}\alpha_i$,
$\alpha_{p+1},\ldots,\alpha_{p+q}$ are in $\{0,1\}$ and $f_\alpha(x,\theta)$ is in
$\F$.

The action of $\Vect(\R^{p|q})$ is induced by its actions on $\mathcal{F}_\lambda$
and $\mathcal{F}_\mu$: it is given by the supercommutator. For every
$D\in\mathcal{D}_{\lambda,\mu}$ and $X\in\Vect(\R^{p|q})$, we have
\[\mathcal{L}_X D=L_X^\mu\circ D -(-1)^{\tilde{X}\tilde{D}}D\circ L_X^\lambda
.\]
This action preserves the filtration of differential operators. The space of
\emph{symbols} is the graded space associated with $\mathcal{D}_{\lambda,\mu}$,
endowed with the Lie derivative induced by the Lie derivative of differential
operators. As a module over $\Vect(\R^{p|q})$, the space of symbols is isomorphic to the space of weighted symmetric tensor fields 
\[\mathcal{S}_{\delta}
=\bigoplus_{l=0}^{\infty}\mathcal{S}^l_{\delta},\quad \delta=\mu-\lambda,\]
endowed with the action of $\Vect(\R^{p|q})$ given in Definition \ref{defsymb}.

The isomorphism is defined by the \emph{principal symbol operator}
\[\sigma_k\colon
\mathcal{D}^k_{\lambda,\mu}\to \mathcal{S}^k_{\delta}\colon D\mapsto \sum_{|\alpha|=
k}f_\alpha(x,\theta)\otimes e_1^{\alpha_1}\vee\cdots \vee e_p^{\alpha_p}\vee
e_{p+1}^{\alpha_{p+1}}\vee\cdots\vee e_{p+q}^{\alpha_{p+q}},\]
where $D$ is given by (\ref{eqDiffop}). This operator commutes with
the action of vector fields and is a bijection from the quotient space
$\mathcal{D}^l_{\lambda,\mu}/\mathcal{D}^{l-1}_{\lambda,\mu}$ to
$\mathcal{S}^l_{\delta}$.
\subsection{Quantizations}
A quantization over $\R^{p|q}$ is a linear bijection $Q$
  from the space of symbols $\mathcal{S}_{\delta}$ to the space 
of differential operators $\mathcal{D}_{\lambda,\mu}$ such
  that
\begin{equation}\label{norma}\sigma_k(Q(T))=T\;\mbox{for all }
T\in\mathcal{S}^{k}_{\delta}\;\mbox{and all }
  k\in\mathbb{N}.\end{equation}
The inverse of such a quantization is a symbol map.
\subsection{The projective superalgebra of vector fields}\label{sl}
Let us recall how to realize the algebra $\mathfrak{pgl}(p+1|q)=\mathfrak{gl}(p+1|q)/\R\Id$ as a subalgebra of
vector fields. This construction is inspired by the definition of the
superprojective space. The functions over the superprojective space of dimension
$p|q$ are superfunctions $g$ of $p+1$ even indeterminates
$(x^0,\ldots,x^p)\in\R^{p+1}\setminus\{0\}$ and $q$ odd indeterminates
$(\theta^1,\ldots,\theta^q)$ that are homogeneous of degree zero in the following
sense: in the expansion of $g$ as
\begin{equation}\label{eq1}
g(x^0,\ldots,x^p,\theta^1,\ldots,\theta^q)=\sum_{I\subseteq\lbrace 1,\ldots,q\rbrace}g_I(x^0,\ldots,x^p)\theta^I,
\end{equation}
($\theta^I=\theta^{i_1}\cdots\theta^{i_r}$ if $i_1<\cdots<i_r$ are the elements of
$I$) each function $g_I$ fulfills the condition
\begin{equation}\label{homog}g_I(r x^0,\ldots,r
x^p)=|r|^{-|I|}g_I(x^0,\ldots,x^p)\quad\mbox{for all } r\in\R\setminus\{0\},\end{equation}
where $|I|$ denotes the cardinality of $I$. This means that the function $g_I$
should be homogeneous of degree $-|I|$. 

Also recall that it is possible to realize the Lie superalgebra $\gl(m|n)$ as a
subalgebra of vector fields of $\R^{m|n}$ by the homomorphism of superalgebras
\[h_{m,n}\colon\mathfrak{gl}(m|n)\to \Vect(\R^{m|n})\colon A\mapsto
-\sum_{i,j}(-1)^{\tilde{j}(\tilde{i}+\tilde{j})}A_j^iy^j\partial_{y^i}.\]
We consider the open subset of $\R^{p+1}$ defined by
$\Omega=\{(x^0,\ldots,x^p)\colon x_0>0\}$ and we denote by $H(\Omega)$
the space of restrictions of functions satisfying (\ref{homog}) to $\Omega$.
There is a correspondence $i\colon\F\to H(\Omega)$
given by 
\[
i(f)(x^0,\ldots,x^p)=\sum_{I\subseteq
\lbrace 1,\ldots,q\rbrace}(x^0)^{-|I|}f_I\left(\frac{x^1}{x^0},\ldots,\frac{x^p}{x^0}\right)\theta^I,\]
if $f$ is given by $\sum_{I\subseteq \lbrace 1,\ldots,q\rbrace} f_I\theta^I$.
Now, the space $H(\Omega)$ is preserved by the action of linear
vector fields. Therefore, we can associate with every linear vector field $X$ over $\Omega$ a vector field $\pi(X)\in \Vect(\R^{p|q})$ by setting
\[\pi(X)(f)=i^{-1}\circ X\circ i(f)\quad\mbox{for all $f$ in }\F.\] 
For a matrix $B=(B^i_j)_{i,j=0,\ldots,p+q}$, the image $\pi\circ h_{p+1,q}(B)$ is given by
\[-\left(\sum_{i,j=1}^{p+q}(-1)^{\tilde{j}(\tilde{i}+\tilde{j})}B_j^i
y^j\partial_{y^i}+\sum_{j=1}^{p+q}B_0^j \partial_{y^j}-\sum_{j=1}^{p+q}B_j^0
y^j(-1)^{\tilde{j}}\euler-B_0^0\euler\right),\]
where $\euler$ is the Euler vector field defined by 
\begin{equation}\label{eq:euler}\euler=\sum_{i=1}^{p+q}y^i\partial_{y^i}.\end{equation}
Moreover, it is easy to check that $\pi\circ h_{p+1,q} (\Id)=0$, so that the
homomorphism $\pi\circ h_{p+1,q}$ induces a homomorphism from $\mathfrak{pgl}(p+1|q)$ to $\Vect(\R^{p|q})$.

For a matrix $B\in \mathfrak{gl}(p+1|q)$, we denote by $[B]$ its equivalence class in $\mathfrak{pgl}(p+1|q)$. The algebra $\mathfrak{pgl}(p+1|q)$ carries a $\Z$-grading defined by
\begin{equation}\label{eq:j}j\colon\mathfrak{pgl}(p+1|q)\to \g_{-1}\oplus \g_0\oplus \g_1\colon\left[\left(\begin{array}{ll}a
&\xi\\h&A\end{array}\right)\right]\mapsto (h,A-a \Id,\xi),
\end{equation}
where $\g_{-1}=\R^{p|q}$, $\g_{0}=\gl(p|q)$ and $\g_{1}=(\R^{p|q})^*$ (the space of
row vectors with entries in $\R$).
Using this notation, the vector fields associated with elements of
$j(\mathfrak{pgl}(p+1|q))$ are given by $h\mapsto X^h$, where
\begin{equation}\label{real}X^h=\left\{\begin{array}{lll}
-\sum_{i=1}^{p+q}v^i\partial_{y^i}&\mbox{if}&h=v\in\R^{p|q}\\
-\sum_{i,j=1}^{p+q}(-1)^{\tilde{j}(\tilde{i}+\tilde{j})}A_j^i
y^j\partial_{y^i}&\mbox{if}&h=A\in \gl(p|q)\\
\sum_{j=1}^{p+q}\xi_j y^j(-1)^{\tilde{j}}\euler&\mbox{if}&h=\xi\in(\R^{p|q})^*.
\end{array}
\right.\end{equation}
It will be useful for our computations to note that, for every $A\in\gl(p|q)$,
Formula (\ref{eqRadouxLeites}) reduces to
\begin{equation}\label{eqXA}L_{X^A}(f\otimes v)=X^A(f)\otimes
v+(-1)^{\tilde{A}\tilde{f}}f\otimes \rho(A)v.\end{equation}

It is also noteworthy that the Euler vector field $\euler$ is exactly $X^{-\Id}$,
where $\Id$ is the identity matrix in $\gl(p|q)$.
By the isomorphism $j^{-1}$, the matrix $-\Id$ corresponds to the equivalence class of matrices
\[\left[\left(\begin{array}{cc}
0&0\\0&-\Id
\end{array}\right)\right]=\left[\left(\begin{array}{cc}
1&0\\0&0
\end{array}\right)\right].\]
We also call this class of matrices the Euler element of $\mathfrak{pgl}(p+1|q)$, and we also denote it by $\euler$. It is completely
determined by its grading property:
\[ad(\euler)|_{\g_k}=k\Id_{\g_k}\quad\mbox{for all }k\in\{-1,0,1\}.\]
When $p$ and $q$ are distinct, we have
$\g_0=\gl(p|q)=\h_0\oplus\R\euler$, where $\h_0=\sla(p|q)$ is the subalgebra of $\g_0$ made of supertraceless
matrices.

Finally, let us recall that $\mathfrak{pgl}(p+1|q)$ is isomorphic to $\sla(p+1|q)$ when $q\not=p+1$, through the map 
\[\iota\colon\mathfrak{pgl}(p+1|q)\to\sla(p+1|q)\colon[A]\mapsto A-\frac{1}{p+1-q}\Id.\]
In what follows, we will always denote by $\sla(p+1|q)$ the algebra $\mathfrak{pgl}(p+1|q)$ when we assume that $q\not=p+1$.
\subsection{Projectively equivariant quantizations}
A projectively equivariant quantization over $\R^{p|q}$ (in the sense of \cite{DLO,LO}) is a
quantization
 \[Q \colon\mathcal{S}_{\delta}\to
\mathcal{D}_{\lambda,\mu}\]
 such that, for every $h\in \mathfrak{pgl}(p+1|q)$, one has
\[\mathcal{L}_{X^h}\circ Q=Q\circ L_{X^h}.\]
The existence and uniqueness of such quantizations in the purely
even case were discussed in
\cite{LO,DO,Lecras} for differential operators acting on densities and in
\cite{BHMP} for differential operators acting on forms.
\section{Construction of the quantization}\label{cons}
Here we will show how the ingredients of the construction of the quantization in the
purely even situation generalize and allow us to build the quantization in the super
situation. A first tool is the so-called \emph{affine quantization map}.
\subsection{The affine quantization map}
Let us build a linear even bijection $Q_{\mathrm{Aff}}$ from the space of symbols to the space of differential operators. This quantization is defined
as the inverse of the total symbol map $\sigma_{\mathrm{Aff}}$ whose restriction to
$\mathcal{D}^k_{\lambda,\mu}$ is
\[\sigma_{\mathrm{Aff}}\colon 
\mathcal{D}^k_{\lambda,\mu}\to \mathcal{S}^k_{\delta}\colon D\mapsto
\sum_{|\alpha|\leqslant k}f_\alpha(x,\theta)\otimes e_1^{\alpha_1}\vee\cdots \vee
e_p^{\alpha_p}\vee e_{p+1}^{\alpha_{p+1}}\vee\cdots\vee e_{p+q}^{\alpha_{p+q}}\]
when $D$ is given by (\ref{eqDiffop}). It is easy to see that this map intertwines the
actions of the affine algebra (made of constant and linear super vector fields) on the space of differential
operators and of symbols.

Now, we can use formula (\ref{real}) in order to express this quantization 
map in a coordinate-free manner:
\begin{proposition}
If $h_1,\ldots, h_k\in\R^{p|q}\cong \g_{-1}$, $t\in
\F$ and 
\[T=t\otimes h_1\vee\cdots\vee h_k,\]
 one has
\[Q_{\mathrm{Aff}}(T) = (-1)^k t\;L_{X^{h_1}}\circ\cdots\circ L_{X^{h_k}}.\]
\end{proposition}
\subsection{The map $\gamma$}
Using the affine quantization map, we can endow the space of symbols with a
structure of representation of $\Vect(\R^{p|q})$, 
isomorphic to $\mathcal{D}_{\lambda,\mu}$. Explicitly, we set
\[\mathcal{L}_XT=Q_{\mathrm{Aff}}^{-1}\circ\mathcal{L}_X\circ 
Q_{\mathrm{Aff}}(T)\]
for every $T\in\mathcal{S}_{\delta}$ and $X\in\Vect(\R^{p|q})$.

Then a projectively equivariant quantization corresponds through the map $Q_{\mathrm{Aff}}$ to
a $\mathfrak{pgl}(p+1|q)-$module isomorphism from
the representation $(\S_{\delta}, L)$ to the representation $(\S_{\delta},\L)$. Also note that the so-defined representation $\L$ on $\S_{\delta}$ depends explicitly on
the parameter $\lambda$ associated with $\mathcal{D}_{\lambda,\mu}$.

In order to measure the difference between these representations, the map 
\[\gamma \colon \g\to \gl(\S_{\delta},\S_{\delta}) \colon h\mapsto \gamma(h)=\L_{X^h}-L_{X^h}\]
was introduced in \cite{BM}. This map can be easily computed in coordinates,
and it has the same properties as in the classical situation (see \cite{BM}): 
\begin{proposition}\label{gamma0}
The map $\gamma$ vanishes on $\g_{-1}\oplus \g_0$. Moreover, for every $h\in\g_1$ and $k\in \N$, the restriction of $\gamma(h)$ to $\S^k_\delta$ has values in $\S^{k-1}_\delta$ and is a differential operator of order zero and
parity $\tilde{h}$ with constant coefficients.
\end{proposition}
\begin{remark}
It follows from Proposition \ref{gamma0} that $\gamma(h)$ is completely determined by its restriction to
constant symbols.
\end{remark}
For our purpose, it will also be interesting to obtain a coordinate-free
expression of $\gamma$. To this aim, we first define the interior product of a row
vector and a symmetric tensor.
\begin{definition}\label{interior}
 The interior product of $h\in(\R^{p|q})^*$ in a symmetric tensor $h_1\vee\cdots\vee
h_k$ $(h_1,\ldots, h_k\in\R^{p|q})$ is defined by
\[i(h)h_1\vee\cdots\vee
h_k=\sum_{j=1}^k(-1)^{\tilde{h}(\sum_{r=1}^{j-1}\tilde{h_r})}\langle h,h_j\rangle
h_1\vee\cdots \widehat{j}\cdots\vee h_k,\] 
where $\langle h,x\rangle$ denotes the standard matrix multiplication of the row
$h$ by the column $x$.
\end{definition}
This means that we extend the pairing of $(\R^{p|q})^*$ and $\R^{p|q}$ as a derivation
of the symmetric product. We may also extend it to $\mathcal{S}_\delta$ by setting
$i(h)u=0$ for $u$ in $B^\delta$ and by defining $i(h)$ as a differential operator of
order zero and parity $\tilde{h}$. This definition allows us to express the operator
$\gamma$ in simple terms.
\begin{proposition}\label{gamma1}
For every $h\in \g_1\cong (\R^{p|q})^*$ we have on ${\mathcal S}^k_\delta$
\[\gamma(h)=-(\lambda(p-q+1)+k-1)i(h),\]
where $i(h)$ is the interior product from Definition \ref{interior}.
\end{proposition}
\begin{proof}
Since both sides of the equality are differential operators of order zero on
${\mathcal S}^k_\delta$, it is sufficient to prove that they yield the same result
when applied to a tensor $h_1\vee\cdots\vee h_k$ ($h_1,\ldots, h_k\in\R^{p|q}$). But
by the definition of $\gamma$, the expression 
\[Q_{\mathrm{Aff}}(\gamma(h)(h_1\vee\cdots\vee h_k))\]
is equal to
\begin{equation}\label{gamma2}\mathcal{L}_{X^{h}}\circ
Q_{\mathrm{Aff}}(h_{1}\vee\cdots\vee h_{k})-
Q_{\mathrm{Aff}}(L_{X^{h}}(h_{1}\vee\cdots\vee h_{k})).\end{equation}
This expression is a differential operator of order at most $k$. Its  term of
order $k$ vanishes: just apply the operator
$\sigma_k$ to (\ref{gamma2}). Hence, we only have to sum up the terms of
order less than or equal to $k-1$ in the first term of expression (\ref{gamma2}).
This latter term writes
\[(-1)^{k}[L_{X^{h}}\circ L_{X^{h_{1}}}\circ\cdots\circ L_{X^{h_{k}}}
-(-1)^{\tilde{h}(\tilde{h_{1}}+\cdots+\tilde{h_{k}})}
 L_{X^{h_{1}}}\circ\cdots\circ L_{X^{h_{k}}}\circ L_{X^{h}}].\]
Using that the Lie derivative defines a representation of the algebra of vector
fields, we can successively move $L_{X^{h}}$ to the right in the first term to see
that the expression is equal to
\[
   (-1)^{k}\sum_{i=1}^k (-1)^{\tilde{h}(\sum_{r=1}^{i-1}\tilde{h_{r}})}
L_{X^{h_1}}\circ\cdots\circ \underbrace{L_{X^{[h,h_i]}}}_{(i)}\circ\cdots \circ
L_{X^{h_k}}.
 \]
Moving $L_{X^{[h,h_i]}}$ back to the left shows that the expression is also equal to
\begin{multline*}
 (-1)^{k}[\sum_{i=1}^k(-1)^{\tilde{h_i}(\sum_{r=1}^{i-1}\tilde{h_{r}})}
L_{X^{[h,h_i]}}\circ
L_{X^{h_1}}\circ\cdots \widehat{i}\cdots \circ L_{X^{h_k}}\\
+ 
\sum_{i=1}^k\sum_{j=1}^{i-1}(-1)^{\tilde{h}(\sum_{r=1}^{i-1}\tilde{h_{r}})+(\tilde{h}+\tilde{h_{i}})(\sum_{r=j+1}^{i-1}\tilde{h_{r}})}
 L_{X^{h_1}}\circ\cdots\underbrace{L_{X^{[h_j,[h,h_i]]}}}_{(j)}\cdots\widehat{i}\cdots\circ
L_{X^{h_k}}].
\end{multline*}
It follows from Equations (\ref{eqRadouxLeites}) and (\ref{eqXA}) that the term of
order less than or equal to $k-1$ in the first summand is exactly
\[(-1)^{k}\sum_{i=1}^k\lambda(p-q+1)(-1)^{\tilde{h_i}(\sum_{r=1}^{i-1}\tilde{h_{r}})}\langle
h,h_i\rangle
L_{X^{h_1}}\circ\cdots\widehat{i}\cdots \circ L_{X^{h_k}}.\]
In order to deal with the second term, we notice that $\langle h,h_{j}\rangle$
vanishes if $h$ and $h_j$ do not have the same parity, we compute the bracket
\[\begin{array}{lll}[h_j,[h,h_i]]&=&\langle h,h_{i}\rangle
h_{j}+(-1)^{\tilde{h_{j}}(\tilde{h_{i}}+\tilde{h})+\tilde{h_i}\tilde{h}}\langle
h,h_{j}\rangle h_{i}\\
   &=&\langle h,h_{i}\rangle h_{j}+(-1)^{\tilde{h}\tilde{h_{j}}}\langle
h,h_{j}\rangle h_{i},
 \end{array}\]
and the result follows.
\end{proof}
\subsection{Casimir operators}
The construction of the quantizations in the even case can be achieved easily by
using Casimir operators. After recalling the definition of second order Casimir operators for Lie superalgebras \cite[p. 57]{KacSketch} (also see \cite{Ber87,Pin90,Mus97,Ser99,SerLei02} and references therein for detailed descriptions of Casimir elements), we specialize to $\mathfrak{pgl}(p+1|q)\cong\sla(p+1|q)$ (thus assuming $q\not=p+1$). We consider the Casimir operators $C$ and $\mathcal{C}$ associated
with the representations $L$ and $\mathcal{L}$ on ${\mathcal S}_\delta$. We show that
the Casimir operator $C$ is diagonalizable and that there is a simple relation
between these operators. As in the purely even case, the analysis of the eigenvector
problem for these operators allows us to build the quantization. 
\begin{definition}
Consider a Lie superalgebra $\mathfrak{l}$ endowed with an even non-degenerate supersymmetric bilinear
form $F$ and a representation $(V,\beta)$ of $\mathfrak{l}$. Choosing a homogeneous basis
$(u_i\colon i\leqslant n)$ of $\mathfrak{l}$ and denoting by $(u'_i\colon i\leqslant n)$ the
$F$-dual basis ($F(u_i,u'_j)=\delta_{i,j}$), the Casimir operator of
$(V,\beta)$ is defined by 
\[C=\sum_{i=1}^n(-1)^{\tilde{u_i}}\beta(u_i)\beta(u'_i)=\sum_{i=1}^n\beta(u'_i)\beta(u_i).\]
\end{definition}
Recall that the Killing form $K_\phi$ associated with a representation $(V,\phi)$ of a Lie superalgebra $\mathfrak{l}$ is defined by
\[K_{\phi}(A,B)=\str(\phi(A)\phi(B))\]
for all $A,B$ in $\mathfrak{l}$. Unless otherwise stated, we will only consider the Killing form associated with the adjoint representation of the superalgebra, which we will denote by $K$ and call the Killing form of the superalgebra. 
Then the Killing form of $\sla(p+1|q)$ is given by 
\begin{equation}\label{Killing}K(A,B)=\str(ad(A)ad(B))=2(p+1-q)\,\str(AB)\end{equation}
for all $A$ and $B$ in $\sla(p+1|q)$. Under our assumption $q\neq p+1$, it is a non-degenerate even supersymmetric bilinear
form on this algebra. 

The Casimir operator does not depend on the choice of a particular basis. We thus choose a basis that is homogeneous with respect to the $\Z$-grading of the algebra. 
\begin{definition}\label{canbases}
 For $r\leqslant p+q$, we denote by $e_r$ (resp. $\varepsilon^r$) the column (resp. row) vector in $\R^{p|q}$ (resp. $(\R^{p|q})^*$) whose entries are 0, except the $r$-th one which is 1.  We also set $\epsilon^r=\frac{(-1)^{\tilde{r}}}{2(p-q+1)}\varepsilon^r$.
\end{definition}
The following proposition is a direct transposition in the super setting of \cite[Proposition 1]{BM}. It can be checked directly.
\begin{proposition}\label{particbases}
For every basis $(g_s,s\leqslant \mathrm{dim}\,\g_0)$ of $\g_0$, the set $(e_r,g_s,\epsilon^t)$ is a basis of $\sla(p+1|q)$ 
whose Killing-dual basis writes $(\epsilon^r,g_s',(-1)^{\tilde{t}}e_t)$, where $(g'_s,s\leqslant \mathrm{dim}\,\g_0)$ is a basis of $\g_0$.
Moreover, we have
\begin{equation}\label{crochet}
\sum_{r=1}^{p+q}(-1)^{\tilde{r}}[e_r,\epsilon^r]=-\frac{1}{2}\euler.
\end{equation}
In the generic situation where $p\not=q$, we can choose a basis of $\g_0$ by selecting a basis $(h_s)$ of $\h_0$ and also the Euler element $\euler$. Then the Killing-dual basis of $(e_r,h_s,\euler,\epsilon^t)$ is $(\epsilon^r,h_s',\frac{\euler}{2(p-q)},(-1)^{\tilde{t}}e_t)$, where $(h'_s)$ is a basis of $\h_0$.
\end{proposition}
Using this particular choice of a basis, we obtain directly a decomposition of the
Casimir operator of any representation of $\sla(p+1|q)$ that is consistent with
the additionnal $\Z$-grading of this algebra (see Formula (\ref{eq:j})).
\begin{lemma}\label{lemma34}
 For any representation $(V,\beta)$ of the algebra $\sla(p+1|q)$, the Casimir
operator of $(V,\beta)$ is given by
\begin{equation}\label{Casim0}
2\sum_{r=1}^{p+q}\beta(\epsilon^r)\beta(e_r)-\frac{1}{2}\beta(\euler)+\sum_s(-1)^{\tilde{g_s}}\beta(g_s)\beta(g'_s).
\end{equation}
When $p\not=q$, this expression can be further developed in
\begin{equation}\label{Casim1}
2\sum_{r=1}^{p+q}\beta(\epsilon^r)\beta(e_r)+\frac{1}{2(p-q)}\beta(\euler)^2-\frac{1}{2}\beta(\euler)+\sum_s(-1)^{\tilde{h_s}}\beta(h_s)\beta(h'_s).
\end{equation}
\end{lemma}
We now introduce a new operator in order to compare the Casimir operators $C$ and
$\cc$.
\begin{definition}\label{defN}
Using the notation of Proposition \ref{particbases}, we define
\[N\colon \mathcal{S}^k_\delta\to\mathcal{S}^{k-1}_\delta \colon S\mapsto
2\sum_i\gamma(\epsilon^i)L_{X^{e_i}}S.\]
\end{definition}
Using Lemma \ref{lemma34} for both $C$ and $\cc$, and recalling that $\L_{X^h}=L_{X^h}$
 for any $h\in\g_{-1}\oplus\g_0$, we obtain directly,
as in \cite{BM}, the following result.
\begin{proposition}\label{relcas}
The Casimir operators are related by
\begin{equation}\label{casinil}\cc=C+N.\end{equation}
\end{proposition}
Now, we want to compute the Casimir operator $C$ on $\mathcal{S}^k_\delta$. We need
to introduce some notation from the classical representation theory of $\gl(p|q)$
or $\sla(p|q)$. For the most part, we follow the notation of \cite{KacLN}.
We denote by $D$ (resp. $H$) the subalgebra of $\gl(p|q)$ (resp. $\sla(p|q)$)
made of diagonal matrices. We denote by $(\eta_i\colon i\leqslant p+q)$ the standard basis
of $D^*$, i.e. if $\Delta=\mathrm{diag}(\Delta_1,\ldots,\Delta_{p+q})$ we set
\[\eta_i(\Delta)=\Delta_i.\] We use the same notation for the restrictions of
$\eta_i$ to $H$. When $p\not=q$, the restriction of the Killing form of $\sla(p|q)$ to $H$ is
non-degenerate. It induces a non-degenerate bilinear form on $H^*$. It can be computed
that this bilinear form is given by
\[(\eta_i,\eta_j)=\frac{(-1)^{\tilde{i}}}{2(p-q)}\delta_{ij}-\frac{1}{2(p-q)^2},\]
for all $1\leqslant i,j\leqslant p+q.$ 
We also consider the Borel subalgebra made of upper triangular matrices. The choice
of this subalgebra allows us to define positive roots. There are also even and odd
roots, corresponding to root spaces in the even or odd subspaces of $\sla(p|q)$.
We denote by $\rho_S$ the element of $H^*$ defined by $\rho_S=\rho_0-\rho_1$, where
$\rho_0$ (resp. $\rho_1$) is half the sum of the positive even (resp. odd) roots. In
this situation, we can show that $\rho_S$ is given by
\begin{equation}
\label{rhos}\rho_S=\frac{1}{2}\left(\sum_{r=1}^p(p-q+1-2r)\eta_r+\sum_{r=1}^q(p+q+1-2r)\eta_{p+r}\right).
\end{equation}
These data allow us to compute the Casimir operator $C$.
\begin{proposition}\label{Cassl}
 The Casimir operator defined by the action of the projective superalgebra
$\mathfrak{pgl}(p+1|q)\cong\sla(p+1|q)$ on $\S^k_\delta$ is a multiple of the identity. The associated
eigenvalue is given by 
\[\alpha(k,\delta)=\frac{p-q}{2}\delta^2-\frac{2k+p-q}{2}\delta+\frac{k(k+(p-q))}{(p-q+1)}.\]

\end{proposition}
\begin{proof}
The proof goes as in the purely even case (see \cite{BM,MR1} for instance). Let us
recall the main developments and assume first that $p$ and $q$ are distinct.
The Casimir operator $C$ commutes with the action of
constant vector fields on $\S^k_\delta$. A direct computation shows that it has
therefore constant coefficients. Hence, in (\ref{Casim1}), we just collect the
terms with constant coefficients. The first summand does not contain such terms if
$\beta=L_X$. In $\beta(\euler)$, terms with constant coefficients are
$(\delta(p-q)-k)$ times the identity. Finally, in the last summand, they are equal to
$\sum_s(-1)^{\tilde{h_s}}\rho(h_s)\rho(h'_s)$, where $\rho$ is the action of
$\sla(p|q)$ on symmetric tensors defined by (\ref{rhotens}). This operator is
actually $\frac{p-q}{p-q+1}C_0$, where $C_0$ is the Casimir operator of the
representation $(S^k,\rho)$ of $\sla(p|q)$. Indeed, the bases $(h_s)$ and
$(h'_s)$ are dual with respect to the Killing form of $\sla(p+1|q)$ and the
restriction of this Killing form to the subalgebra $\sla(p|q)$ is equal to
$(p-q+1)/(p-q)$ times the Killing form of the subalgebra. Since the space
$S^k$ is an irreducible representation with highest weight $k\eta_1$, its Casimir operator is a multiple of the identity with eigenvalue 
\begin{equation}\label{value}(k\eta_1,k\eta_1+2\rho_S)=\frac{k(p-q-1)}{2(p-q)^2}(k+(p-q)),\end{equation}
and the result follows in this case. 

In the particular case $p=q$, we use the same developments, but we use (\ref{Casim0}) instead of (\ref{Casim1}) to obtain 
\begin{equation}\label{cpp}
C=\frac{1}{2}\rho(\Id)+\sum_s(-1)^{\tilde{g_s}}\rho(g_s)\rho(g'_s).
\end{equation}
The first summand is clearly equal to $\frac{k}{2}\Id$, while the second is the Casimir operator of $\gl(p|p)$ associated with the invariant bilinear form $K_0$ induced on $\g_0=\gl(p|p)$ by the Killing form of $\sla(p+1|p)$. This form is given by
\[K_0\colon\gl(p|p)\times\gl(p|p)\to\R\colon(A,B)\mapsto 2(\str(AB)-\str(A)\str(B)).\]
This Casimir operator can be computed in a standard way: we notice that the space of weighted symmetric tensors $S^k_\delta$ over $\R^{p|p}$ is a representation of $\gl(p|p)$ with highest weight 
\[\Lambda=k\eta_1-\delta \str=k\eta_1-\delta\left(\sum_{r=1}^p(\eta_r-\eta_{r+p})\right).\]
The restriction of $K_0$ to the Cartan subalgebra $D$ of diagonal matrices is non-degenerate and allows one to define a bilinear form $(\cdot,\cdot)$ on $D^*$. It is then easy to compute the relations
\[(\Lambda,\Lambda)=k(k-\delta)\quad\mbox{and}\quad (\Lambda,2\rho_S)=-\frac{k}{2},\]
where $\rho_S$ is given by (\ref{rhos}) (with $q=p$). Therefore, the last summand in (\ref{cpp}) reads $(k(k-\delta)-\frac{k}{2}) \Id$ and the result follows.
\end{proof}
\section{The explicit construction}
\subsection{Critical values}
Let us start this section with the definition of critical values of the parameter $\delta$. 
\begin{definition}
A value of $\delta$ is \emph{critical} if there exist $k,l\in\N$ with $l<k$ such
that $\alpha(k,\delta)=\alpha(l,\delta)$. 
\end{definition}
Using Proposition \ref{Cassl}, we can easily describe the set of critical values.
\begin{proposition}\label{gencrit}
 The set of critical values for the $\sla(p+1|q)$-equivariant quantization over $\R^{p|q}$ is given by
\[\mathfrak{C}=\cup_{k=1}^\infty\mathfrak{C}_k,\quad\mbox{where}\quad\mathfrak{C}_k=\left\{\frac{2k-l+p-q}{p-q+1}\colon l=1,\ldots, k\right\}.\]
\end{proposition}
\begin{remark}
The critical values here correspond to the ones in the classical situation (see \cite[p.~63]{DO}) up to replacement of the dimension $n$ by the superdimension $p-q$. However, since the superdimension can be negative, there exist some values of $(p|q)$ such that $\delta=0$ is a critical value.
\end{remark}
\subsection{The construction}
We can now state the main results.
\begin{theorem}\label{flatex}
If $\delta$ is not critical, then there exists a unique projectively equivariant
quantization from $\mathcal{S}_{\delta}$ to $\mathcal{D}_{\lambda,\mu}$.
\end{theorem}
\begin{proof}
The proof is as in \cite{BM} and \cite{DLO}. We give here the main ideas for
the sake of completeness.

First, remark that for every $S\in\mathcal{S}_{\delta}^{k}$, there exists a unique
eigenvector $\hat{S}$ of $\cc$ with eigenvalue $\alpha(k,\delta)$ such that 
\[\left\{\begin{array}{l}\hat{S}= S_k + S_{k-1}+\cdots+S_0,\quad S_k=S\\
S_l\in\mathcal{S}_{\delta}^{l}\quad\mbox{for all } l\leqslant
k-1.\end{array}\right.\]
Indeed, these conditions write 
\begin{equation}\label{flatP}\left\{\begin{array}{l}
C(S)=\alpha(k,\delta)S\\
(C-\alpha(k,\delta)\mbox{Id})S_{k-l}=- N(S_{k-l+1})\quad\mbox{for all }
l\in\{1,\ldots,k\}\\
S_{k-l}\in\mathcal{S}_{\delta}^{k-l}.
\end{array}\right.\end{equation}
This system of equations has a unique solution. Indeed, as $\delta$ is
not critical, the differences $\alpha(k,\delta)-\alpha(l,\delta)$ are different from
$0$.

Now, define the quantization $Q$ (remark that this is the only way to proceed) by
\[Q\vert_{\mathcal{S}_{\delta}^{k}}(S)=\hat{S}.\]
It is clearly a bijection and it also fulfills 
\[Q\circ L_{X^h} = \L_{X^h}\circ Q\quad\mbox{for all } h\in\,\sla(p+1|q).\]
Indeed, for all $S\in\mathcal{S}_{\delta}^{k}$, the tensors $Q(L_{X^h}S)$ and $ \L_{X^h}(Q(S))$
share the following properties:
\begin{itemize}
\item they are eigenvectors of $\cc$ of eigenvalue $\alpha(k,\delta)$ because, on
  the one hand, $\cc$ commutes with $\L_{X^h}$ for all $h$ and, on the other
  hand, $C$ commutes with $L_{X^{h}}$ for all $h$.
\item their term of degree $k$ is exactly $L_{X^h}S$.
\end{itemize}
The first part of the proof shows that they have to coincide.
\end{proof}
Let us now introduce a new operator on symmetric tensor fields.
\begin{definition}Using the bases of Definition \ref{canbases} we define the divergence operator on symbols by~:
\begin{equation}\label{divergence}\dive \colon \S^k_{\delta} \to \S^{k-1}_{\delta} \colon S\mapsto \sum_{j=1}^{p+q}
(-1)^{\tilde{y^j}}i(\varepsilon^{j})\partial_{y^j}S.\end{equation}
\end{definition}
\begin{remark}\label{rem:div}
It is easy to check that the restriction of this divergence operator to vector
fields, that is, to $\S^1_0$, coincides with the divergence of
Definition~\ref{Defdiv}. It is also easy to check that this divergence operator is not invariant
 with respect to divergence-free vector fields unless $k\leqslant 1$.
\end{remark}
We can now give the explicit formula for the $\sla(p+1|q)$-equivariant quantization.
\begin{theorem}\label{Expl} If $\delta$ is not critical, then the map $Q:
\mathcal{S}_{\delta}\to
\mathcal{D}_{\lambda,\mu}$ defined by
\begin{equation}\label{formula}Q(S)(f) = \sum_{r=0}^k
C_{k,r}Q_{\mathrm{Aff}}(\dive^{r}S)(f)\quad\mbox{for all }
S\in\mathcal{S}_{\delta}^k\end{equation}
is the unique $\sla(p+1|q)$-equivariant quantization over $\R^{p|q}$ if
\[C_{k,r} =\frac{\prod_{j=1}^r((p-q+1)\lambda + k-j)}{r!\,\prod_{j=1}^r
(p-q+2k-j -(p-q+1)\delta)}\mbox{ for all } r\geqslant 1,\quad C_{k,0}=1.\]
\end{theorem}
\begin{proof}
Using the definition of $N$, Equation (\ref{flatP}) is equivalent to 
\[S_{r}=\frac {2\sum_{i=1}^{p+q}
\gamma(\epsilon^i)L_{X^{e_i}}S_{r+1}}{\alpha(k,\delta)-\alpha(r,\delta)},\quad
0\leqslant r\leqslant k-1.\]
We conclude
using Definition \ref{canbases} and Proposition \ref{gamma1}.
\end{proof}
\begin{remark}\label{rmk:5}\hspace{\fill}
\begin{itemize}
\item[a)] Formula (\ref{formula}) coincides with Formula (2.4) given in \cite{DO}, if we replace the dimension $n$ by the superdimension $p-q$ and the divergence operator $D$ by the divergence operator $\dive$ defined above. In particular, Formula (\ref{formula}) can be written under the same form as Equation (3.2) in \cite{DO}, using hypergeometric functions.
\item[b)] It was already pointed out in \cite{Lecras} and \cite{DO} that for critical values of $\delta$ the quantization does not exist for generic values of $\lambda$, but there are some specific values of this parameter such that the equivariant quantization exists. 
 \end{itemize}
\end{remark}
\section{$\mathfrak{psl}(p+1|p+1)$- and $\mathfrak{pgl}(p+1|p+1)$-equivariant quantizations}
\subsection{The setting}
In the developments above, we analyzed the $\mathfrak{pgl}(p+1|q)$-equivariant quantizations
on $\R^{p|q}$ when $q\not=p+1$. In this section, we analyze the remaining special 
case that presents phenomena that have no analogues in the classical theory of projectively equivariant quantization. 

Indeed, the projective superalgebra $\mathfrak{pgl}(p+1|p+1)$ is not isomorphic to $\sla(p+1|p+1)$. The formula $\dive(f X)=f\dive X+X(f)$ that holds for every function $f$ and every even vector field $X$, shows that the quadratic vector fields of this superalgebra are divergence-free.
It is not endowed with a non-degenerate bilinear symmetric invariant form and moreover it is not simple,
 because it has a codimension one ideal $\mathfrak{psl}(p+1|p+1)$. We now analyze the existence of the $\mathfrak{psl}(p+1|p+1)$-equivariant quantization. Since this algebra is made of divergence-free fields, the modules of densities are all isomorphic to the space of functions and the existence of the quantization does not depend on the parameters $\lambda$ and $\delta$. 
We show that the method above can be applied to the ideal $\mathfrak{psl}(p+1|p+1)$ to define a one-parameter family of $\mathfrak{psl}(p+1|p+1)$-equivariant quantizations over $\R^{p|p+1}$ and that these quantizations are in fact $\mathfrak{pgl}(p+1|p+1)$-equivariant.
 
It was remarked by I. Kaplansky \cite{Kaplan} that, even though the Killing form of $\mathfrak{psl}(p+1|p+1)$ vanishes, this algebra is endowed with a non-degenerate invariant supersymmetric even form, namely
\[\mathcal{K}([A],[B])=\str(AB).\]
The isomorphism $j$ (see (\ref{eq:j})) identifies the subalgebra $\mathfrak{psl}(p+1|p+1)$ to ${\g_{-1}\oplus\g_{0}\oplus\g_{1}}$, where $\g_{-1}=\R^{p|p+1}$, $\g_{0}=\sla(p|p+1)$ and $\g_{1}=(\R^{p|p+1})^*$. 
\subsection{Construction of the quantization}
We now modify the arguments of 
Section \ref{cons} to get the existence of the quantization. 
\begin{theorem}\label{thmcaspsl}
The Casimir operator $C$ of $\mathfrak{psl}(p+1|p+1)$ acting on $\S^k_\delta$ is equal to $2k(k-1)$ times the identity.
\end{theorem}
\begin{proof}
We choose a suitable basis as above by considering $(e_{i}, g_s, (-1)^{\tilde{i}}\varepsilon^{i})$, where $e_i$ and $\varepsilon^{i}$ are as in Definition \ref{canbases} and the elements $g_s$ form a basis in $\sla(p|p+1)$. It is then easy to see that the $\mathcal{K}$-dual basis is of the form $((-1)^{\tilde{i}}\varepsilon^{i},g_s',(-1)^{\tilde{i}}e_{i})$, where $g_s'$ is some basis in $\sla(p|p+1)$. An easy computation shows that, in $\mathfrak{psl}(p+1|p+1)$, we have 
\[\sum_{i}[e_{i},\varepsilon^{i}]=0.\]
Therefore, we can then adapt Lemma \ref{lemma34}: for any representation $(V,\beta)$ of $\mathfrak{psl}(p+1|p+1)$, the Casimir operator of $(V,\beta)$ is given by
\begin{equation}
2\sum_{r=1}^{2p+1}\beta((-1)^{\tilde{r}}\varepsilon^r)\beta(e_r)+\sum_s(-1)^{\tilde{g_s}}\beta(g_s)\beta(g'_s).
\end{equation}
When $\beta=L_{X}$, we collect again the terms with constant coefficients. The first summand does not contain such terms and in the second summand, they are equal to 
\begin{equation}\label{coucou}\sum_s(-1)^{\tilde{g_s}}\rho(g_s)\rho(g'_s),\end{equation}
where $\rho$ denotes the action of $\sla(p|p+1)$ on supersymmetric tensors (see Formula (\ref{rhotens})). We remark that the Killing form of $\sla(p|p+1)$ is equal to $-2$ times the restriction of the form $\mathcal{K}$ to $\sla(p|p+1)$, so that (\ref{coucou}) is equal to $-2$ times the Casimir operator of $\sla(p|p+1)$. The result follows using (\ref{value}), with $q=p+1$.
\end{proof}
We deduce from the previous result that the situation is critical for the $\mathfrak{psl}(p+1|p+1)$-equivariant quantization, since the eigenvalues of the Casimir operator on $\S_\delta^1$ and on $\S_\delta^0$ coincide. However, it is possible to build the quantization because $\gamma(h)$ vanishes on $\S_\delta^1$ for every $h\in \g_1$ (see Remark \ref{rmk:5}).
\begin{theorem}\label{expl2}
The map $Q \colon\mathcal{S}_\delta\to \mathcal{D}_{\lambda,\mu}$
given on $\mathcal{S}^k_\delta$ for $k\not=1$ by Theorem \ref{Expl} and
by
\[Q_1 \colon \mathcal{S}^1_\delta\to \mathcal{D}_{\lambda,\mu}\colon
S\mapsto Q(S)(f)=Q_{\mathrm{Aff}}(S)f + t \,\dive(S) f\]defines a
$\mathfrak{psl}(p+1|p+1)$-equivariant quantization over $\R^{p|p+1}$ for every $t\in\R$. 
\end{theorem}
\begin{proof}
For the quantization of symbols of degree not equal to 1, we just follow the lines of Theorems \ref{flatex} and \ref{Expl}. For symbols of degree 1, since $\gamma(h)$ vanishes on $\S_\delta^1$, $Q_{\mathrm{Aff}}$ defines a $\mathfrak{psl}(p+1|p+1)$-equivariant quantization of symbols of degree 1 (recall that $\gamma$ precisely measures the failure of equivariance of $Q_{\mathrm{Aff}})$. The proposed map $Q_1$ defines a quantization for every $t$, since the $\mathfrak{psl}(p+1|p+1)$-modules 
$\S_\delta^1$ and $\S_0^1$ are equivalent and we have therefore, using the basic properties of the divergence, 
\[L_X\dive S=\dive L_X S\] for every $S\in S_\delta^1$ and every divergence-free vector field $X$. 
%
\end{proof}
We now have a result about the $\mathfrak{pgl}(p+1|p+1)$-equivariant quantization.
\begin{theorem}\label{expl3}
The $\mathfrak{psl}(p+1|p+1)$-equivariant quantizations defined in Theorem \ref{expl2} are equivariant with respect to
$\mathfrak{pgl}(p+1|p+1)$. 
\end{theorem}
\begin{proof}
 Since the algebra of vector fields $\mathfrak{pgl}(p+1|p+1)$ is generated by $\mathfrak{psl}(p+1|p+1)$ and the Euler field $\euler=\sum_{i=1}^{2p+1}y^i\partial_{y^i}$, it is sufficient to show that the quantization is equivariant with respect to this linear vector field. This reduces to show that we have
\[L_{\euler} \dive S=\dive L_{\euler}S\]
for every $S\in \S^k_\delta$. This can be checked directly.
\end{proof}
\nocite{KacAdvances}
\section{Acknowledgments}
It is a pleasure to thank V. Ovsienko for fruitful discussions, P. Lecomte and J.-P. Schneiders for their interest in our work.
We express our gratitude to the referees of this paper for their valuable comments and suggestions.
P.~Mathonet is supported by the University of Luxembourg internal research project
F1R-MTH-PUL-09MRDO.
F. Radoux thanks the Belgian FNRS for his research fellowship.


\end{document}